\documentclass[11pt]{amsproc}
\usepackage{graphicx, amssymb,amsmath, latexsym,cite}
\usepackage{amsthm, xcolor}
\usepackage{lscape}  
\usepackage{float}
\usepackage{fancyhdr, lastpage} 
\usepackage{times,url}    
\usepackage{color,multicol}

 \usepackage[a4paper,right=2.4cm, left=2.4cm, top=2.3cm, bottom=2.3cm, marginpar=1.6cm]{geometry}  

\newtheoremstyle{mythm}                   
{7pt}
{0pt}
{\it}
{}
{\bf}
{.}
{.5em}
{}

\newtheoremstyle{mydef}                   
{6pt}
{6pt}
{}
{}
{\bf}
{.}
{.5em}
{}

\newtheoremstyle{myrem}                   
{6pt}
{6pt}
{}
{}
{\bf}
{.}
{.5em}
{}

\theoremstyle{mythm}       

\newtheorem*{theorem*}{Theorem}

\theoremstyle{mydef}

\theoremstyle{myrem}

\numberwithin{equation}{section}

\newcommand{\M}{\mathbb{M}}

\newcounter{ithmcount}

\newcounter{itemscount}

\newcommand{\Sym}{\mathrm{Sym}}

\newcommand{\Dih}{\mathrm{Dih}}
\newcommand{\Co}{\mathrm{Co}}
\newcommand{\Gx}{G_{x0}}

\newcommand{\N}{N_{\M}}
\newcommand{\C}{C_{\M}}

\renewcommand{\leq}{\leqslant}

\allowdisplaybreaks

\begin{document}

\vspace*{-2.2cm}

\title{The Monster group is a completion of the Goldschmidt $G_3$-amalgam}

\begin{abstract}
  We prove the claim in the title and complete Parker and Rowley's classification (J.~Algebra 2001) of which sporadic simple  groups are completions of the Goldschmidt $G_3$-amalgam.
\end{abstract} 
 
\author[H.~Dietrich]{Heiko Dietrich}
\address{School of Mathematics, Monash University, Clayton, Australia}
\email{heiko.dietrich@monash.edu}
\date{\today}

\thanks{The author thanks L\'aszl\'o Pyber for raising this question and Martin Seysen for comments on the draft.}

\maketitle
  


A group $G$ is a completion of the Goldschmidt $G_3$-amalgam if it has subgroups $A,B\leq G$, both isomorphic to $\Sym_4$, that generate $G$  and satisfy $A\cap B\cong \Dih_8$ and $O_2(A)\ne O_2(B)$. The study of these (and other) amalgams stems from the influential work of  Goldschmidt \cite{goldschmidt} on automorphisms of trivalent graphs. We refer to the recent work of Rowley and  Vasey \cite{rowley} for background information and a survey of known results. Here we are concerned with the sporadic simple groups: Parker and Rowley  \cite{parker} proved that $15$ of the $26$ sporadic simple groups are completions of $G_3$-amalgams, and $10$ are not. The case of the Monster group $\M$ was left open; we settle it here.

\begin{theorem*}
The Monster is a completion of the Goldschmidt $G_3$-amalgam.
\end{theorem*} 
\begin{proof}
In  Figure \ref{fig:gens} we provide code for Seysen's Python software package \texttt{mmgroup} \cite{mmgroup,fast_monster} for computing in $\M$. Specifically, the code  defines generators for subgroups $A,B<\M$ and proves that $A\cap B\cong \Dih_8$ and  $A\cong B\cong \Sym_4$. The code also defines elements $u,v\in\langle A,B\rangle$  of orders $71$ and $47$, respectively. The classification of the maximal subgroups of $\M$ shows that, up to conjugacy, there is a unique maximal subgroup whose order is divisible by $71$, namely $L={\rm PSL}_2(71)$, see \cite[Table 1]{dlp}. The size of $L$ is not divisible by $47$, so $\langle u,v\rangle=\M$. Thus, $\M=\langle A,B\rangle$, and $O_2(A)\ne O_2(B)$ by simplicity of $\M$.
\end{proof}
 
We briefly comment on the random search that led to our generators. The package \texttt{mmgroup} \cite{mmgroup} can compute particularly well in the maximal subgroup $\Gx = 2_+^{1+24}.\Co_1$ of $\M$, which is the centraliser of a $2{\rm B}$-involution. For this reason, we constructed the first $\Sym_4$ in $\Gx$, by first looking for random  involutions that generate a subgroup $\Sym_3$ in the quotient $\Co_1$, and then considering suitable preimages that generate a subgroup $A\cong \Sym_4$ in $\Gx$. This search yielded the generators $a_1$ and $a_2$ in Figure \ref{fig:gens}. We construct $D\cong \Dih_8$ as a specific Sylow $2$-subgroup of $A$. Note that $D$ has exactly two subgroups isomorphic to $C_2\times C_2$: one is the  $2$-core $O_2(A)<D$, the other, $W<D$, is not normal in $A$.

Next, we look for a second group $B\cong \Sym_4$ that satisfies $A\cap B=D$ and $\M=\langle A,B\rangle$. The latter forces $O_2(B)\ne O_2(A)$ since otherwise $O_2(A)$ would be normal in $\M$, which is not possible. Thus, we need $O_2(B)=W$. Any such $B$ lies in $\N(W)$ and defines a subgroup $\Sym_3$ in $\N(W)/W$; we  construct  $B$ as follows. Conjugation by an involution $i\in D\setminus W$ induces a transposition on the involutions in  $W$. Let $h\in \N(W)$ and define $w=h^{-1}h^i$, so that $w^i=w^{-1}$. If $h$ induces a transposition on the involutions in $W$ that is different from the one induced by $i$, then $w$ induces a $3$-cycle on $W$. We look for such elements $w$ until we find one such that $b_2=w^{|w|/3}$ satisfies $B=\langle D,b_2\rangle\cong \Sym_4$. How to find $h$? The three involutions $z_1,z_2,z_3$ in our group $W$ all lie in the $\M$-class $2{\rm B}$, and we can use \texttt{mmgroup} functionality to construct  elements $h$ that fix $z_1$ and swap $z_2$ and $z_3$:  the command \texttt{conjugate\_involution} yields $c\in\M$ such that $z_1^c$ is the central involution in $\Gx$, so $W^c=\{1,z_1^c,z_2^c,z_3^c\}\leq \Gx$. Then \texttt{conjugate\_involution\_G\_x0} yields $d,e\in\Gx$ that map $z_2^c$ and $z_3^c$ to the same standard class representative in $\Gx$, so that $z_2^{cd}=z_3^{ce}$; note that $d$ and $e$ necessarily centralise $z_1^c$. Now $h=cde^{-1}c^{-1}\in \M$ fixes $z_1$ and swaps $z_2=z_1z_3$ and $z_3=z_1z_2$. Further  elements with this property can be constructed by multiplying any such $h$ by random elements of $\C(W)=C_{\C(z_1)}(z_2)$,  obtained   by Bray's method \cite{bray}.   

We used the LLM Claude Opus to significantly improve the Python implementation for the random search described above (e.g., by using multiprocessing); with the final code,  a successful construction took about 5 minutes on a 2019 iMac with a 3GHz Intel i5 processor.

\begin{figure}[pt]
\scriptsize
\begin{verbatim}
from mmgroup import MM

## defining the generators of the two subgroups Sym_4:
a1 = MM("M<y_75dh*x_18b7h*d_628h*p_56998179*l_2*p_2956800*l_1*p_86279601*l_2*p_931200>")                                
a2 = MM("M<y_0e6h*x_1819h*d_734h*p_240545179*l_2*p_2597760*l_1*p_86277665*l_2*p_10707840>")
b1 = MM("M<x_13e6h*d_2c7h>")                                     
b2 = MM("M<y_48h*x_60bh*d_0f13h*p_202245973*l_1*p_1015680*t_2*l_1*p_2027520"
        "*l_1*p_170653255*t_1*l_2*p_2597760*l_1*p_42664523*t_1*l_1*p_107817600>")                                   

## proof that A=<a1,a2> and B=<b1,b2> are isomorphic to Sym_4;
## note that Sym_4 = < g,h | g^2, h^3, (gh)^4 >:
assert a1.order() == 2 and a2.order() == 3 and (a1*a2).order() == 4
assert b1.order() == 2 and b2.order() == 3 and (b1*b2).order() == 4

## proof that D = <m,r> is isomorphic to Dih_8;
## note that Dih_8 = < m,r | m^2, r^4, r^m = r^3 >:
r  = a1 * a2
m  = (a2 * a1)**2
assert r.order() == 4 and m.order() == 2 and r**m  == r**3

## D lies in A by construction; proof that D lies in B:
assert r == b2 * b1 * b2 and m == b1 * b2 * b1 * b2**-1 * b1 * b2

## proof that M = <A,B>; this also shows that A <> B:
## the elements u, v lie in <A,B>; they have orders 71 and 47, respectively,
## so M = < u,v > by the classification of the maximal subgroups of M:
u =  a2**-1 * b2**-1 * a2**-1 * b2**2 * a1 * b1 * b2**-1
v =  b2 * a2 * b2**-1 * a2 * b2**-1 * b1 * a1 * b2
assert u.order() == 71 and v.order() == 47
\end{verbatim}
\caption{Python code for the proof of the Theorem, using \texttt{mmgroup} version 1.0.7.}
\label{fig:gens}
\end{figure}



\begin{thebibliography}{9}

\bibitem{bray}
J.~N.\ Bray.
An improved method for generating the centralizer of an involution.
Arch.\ Math.\ 74:241--245, 2000.


\bibitem{dlp}
H.~Dietrich, M.~Lee, T.~Popiel. The maximal subgroups of the Monster.
Adv.\ Math.\ 469:110214,  2025.

\bibitem{goldschmidt}
D.~M. Goldschmidt.
Automorphisms of trivalent graphs.
Ann.\ of Math.\ 111:377--406, 1980. 
 
\bibitem{parker}
C.\ Parker, P.\ Rowley. Sporadic simple groups which are completions of the Goldschmidt $G_3$-amalgam. J.\ Algebra 235(1):131--153, 2001.


\bibitem{rowley}
P.\ J.\ Rowley, D.\ M.\ Vasey.  
Completions of the Goldschmidt  $G_3$-amalgam and alternating groups.
Comm.\ Algebra 51:1186--1200, 2023.

\bibitem{fast_monster}
M.~Seysen.
A fast implementation of the Monster group.
J.\ Comput.\ Alg.\ 9:100012, 2024.

\bibitem{mmgroup}
M.~Seysen.
The \texttt{mmgroup} package, version 1.0.7.
{\url{https://github.com/Martin-Seysen/mmgroup}}.


\end{thebibliography}
\end{document}